\documentclass{article}
\usepackage[utf8]{inputenc}
\usepackage{amsthm}
\usepackage{amsmath}
\usepackage{amsfonts}
\usepackage{latexsym}
\usepackage{geometry}
\usepackage{enumerate}
\usepackage{csquotes}
\usepackage{graphicx}
\usepackage{comment}
\usepackage{appendix}
\usepackage{hyperref}
\usepackage{bm}

\emergencystretch=\maxdimen
\def \dist {{\rm dist}}

\DeclareMathOperator{\Ric}{Ric}

\makeatletter
\newcommand*{\rom}[1]{\rm {\expandafter\@slowromancap\romannumeral #1@}}
\makeatother

\numberwithin{equation}{section}
\newtheorem{Theorem}{Theorem}[section]

\newtheorem{Lemma}[Theorem]{Lemma}

\theoremstyle{definition}

\title{An optimal volume growth estimate for noncollapsed steady gradient Ricci solitons}
\author{Richard H. Bamler\footnote{supported by NSF grant DMS-1906500.}, Pak-Yeung Chan, Zilu Ma, Yongjia Zhang}
\date{}

\begin{document}

\maketitle

\begin{abstract}
    In this paper, we prove a volume growth estimate for steady gradient Ricci solitons with bounded Nash entropy. We show that such a steady gradient Ricci soliton has volume growth rate no smaller than $r^{\frac{n+1}{2}}.$
    This result not only improves the estimate in \cite[Theorem 1.3]{CMZ21b}, but also is optimal since the Bryant soliton and Appleton's solitons \cite{Ap17} have exactly this growth rate.
\end{abstract}

\section{Introduction}

The Ricci flow has been a powerful tool in settling various longstanding problems in geometry and topology, among which the most well-known ones are the geometrization and the Poincar\'e conjectures. The success of the Hamilton--Perelman program \cite{Ha95, Per02, Per03a, Per03b} in dimension $3$ suggests that the analysis of singularity formation plays a central role in the study of the Ricci flow. In the Hamilton--Perelman program, a \emph{singularity model} is understood to be an ancient solution arising as the smooth limit of a scaled sequence of a Ricci flow forming a finite-time singularity  (see below for more details). Among all singularity models, the most important ones are the shrinking and steady \emph{gradient Ricci solitons}. Perelman's canonical neighborhood theorem shows that a $3$-dimensional Ricci flow becomes locally close to a singularity model wherever the curvature is large. However,  due to the lack of the Hamilton-Ivey pinching estimate, this canonical neighborhood theorem is generally not true in higher dimensions. 

Recently, the first-named author \cite{Bam20a,Bam20b,Bam20c} established a new theory about weak limits of Ricci flows on closed manifolds. This theory sheds more light on the formation of singularities in dimension $4$ and higher. Indeed, the  three last-named authors have already employed these methods in the study of ancient solutions and singularities of the Ricci flow; see, for instance, \cite{MZ21, CMZ21a,CMZ21b,CMZ21c}. Very recently, the first-named author \cite{Bam21} proved that the fundamental group of a noncollapsed ancient Ricci flow is finite. In this paper, we shall study the volume growth of steady gradient Ricci solitons using these techniques.


Let us recall the definition of gradient Ricci soliton. A triple $(M^n,g,f)$ is a called a gradient Ricci soliton if 
\begin{equation}\label{generalsoliton}
    \Ric + \nabla^2 f = \tfrac{\kappa}{2} g,
\end{equation}
for some constant $\kappa\in \mathbb{R}.$
The soliton is called \emph{shrinking} if $\kappa>0$, \emph{steady} if $\kappa=0$, and  \emph{expanding} if $\kappa<0.$
Any soliton canonically induces a Ricci flow, called the \emph{canonical form}. Precisely, if we define $\Phi_t$ and $g_t$ by
\begin{align}\label{canonicalform}
    \frac{\partial}{\partial t}\Phi_t & =\frac{1}{1-\kappa t}\nabla f\circ \Phi_t,\\\nonumber
      \Phi_0&= \text{ id},\\\nonumber
    g_t &= (1-\kappa t)\Phi_t^*g,
\end{align}
then $g_t$ moves by the Ricci flow. In the shrinking ($\kappa>0$) and steady ($\kappa=0$) case, the canonical form is not only a self-similar ancient solution moving by diffeomorphism, but also often arises as a singularity model. For instance, the blow-up limit at every Type I singularity is (the canonical form of) a shrinking gradient Ricci soliton (c.f. \cite{EMT11}), and a degenerate neck-pinch (c.f. \cite{GZ08}) is modeled on a Bryant soliton. We remark that by the recent work of Choi-Haslhofer \cite{CH21}, if we consider the more general \emph{singular Ricci flow} (c.f. \cite{KL17,BK17}) instead of Ricci flow, then there could possibly be non-solitonic blow-up limits.

The study of steady gradient Ricci solitons is important not only for the understanding of the formation of Type-II singularities in particular, but also for the understanding of ancient Ricci flows in general. For instance, a steady soliton may arise as a sequential limit from a shrinking soliton with exactly quadratic curvature growth (c.f. \cite{CFSZ20}); the only positively curved ancient, noncompact and noncollapsed Ricci flow in dimension $3$ is the Bryant soliton (c.f. \cite{Br20}). 

Unlike shrinking solitons, though, the geometric characterizations of steady solitons are less complete. This, to some extent, is reflected by the newer examples constructed by Appleton \cite{Ap17} and Lai \cite{Lai20}. Furthermore, shrinking solitons are automatically strongly noncollapsed (c.f. \cite{CN09, LW20}), but this is obviously not true for steady solitons. In fact, the cigar soliton of Hamilton---the first steady soliton ever found---and the $3$-dimensional flying wings of Lai \cite{Lai20}, conjectured by Hamilton, are collapsed. Up to this point, the volume growth estimates of steady solitons are also less sharp than that of shrinking solitons. O. Munteanu and J.P. Wang \cite{MW12} showed that the volume of the geodesic ball of a noncompact gradient shrinker grows at least linearly in the radius, i.e., $|B(p,r)|\ge Cr$, where $C=C(n)e^{c(n)\mu}$ depends on the dimension and the shrinker entropy $\mu$, and $c(n)>1.$ By using their Sobolev inequality, Y. Li and B. Wang \cite[Proposition 6]{LW20} provided a better constant $C=c(n)e^{\mu}$. 
This estimate is optimal since it is satisfied by cylinders. However, the same technique does not yield an equally nice volume growth estimate for steady solitons. Indeed, the three last named authors \cite{CMZ21b} proved that a Sobolev inequality on a steady soliton implies that the volume growth rate is at least $r^{\frac{n}{2}}$, but this is not optimal since the Bryant soliton has volume growth rate $r^{\frac{n+1}{2}}$.

In this paper, we prove an optimal volume growth estimate for steady gradient Ricci solitons with bounded Nash entropy. First of all, we recall some known results on the volume growth rate for steady solitons. Besides the $r^{\frac{n}{2}}$ volume growth rate lower bound mentioned above (c.f. \cite{CMZ21b}), Munteanu-\v{S}e\v{s}um \cite{MS13} showed that a steady soliton has at least linear volume growth, Cui \cite{Cui16} proved a  volume growth lower bound for steady K\"{a}hler Ricci solitons with positive Ricci curvature.
The \textit{optimal} volume growth lower bound proved in this paper says that a steady gradient Ricci soliton with \emph{bounded Nash entropy} has volume growth rate no smaller than $r^{\frac{n+1}{2}}$. Since the Bryant soliton (c.f. \cite{Cao09}) as well as Appleton's solitons (\cite{Ap17}, they are asymptotic to quotients of the Bryant soliton) have exactly this volume growth rate, our result is optimal indeed. As a consequence, a steady gradient Ricci soliton with volume growth strictly slower than $r^{\frac{n+1}{2}}$ cannot arise as a singularity model (see below).

Throughout the paper, we shall assume that
$(M^n,g,f)$ is a complete steady gradient Ricci soliton normalized in the way that
\begin{equation}
\label{eq: steady def}
   \Ric = \nabla^2 f,\quad
    R+|\nabla f|^2=1. 
\end{equation}
Here, for the notational simplicity, we have reversed the sign of $f$ in (\ref{generalsoliton}). Then the $1$-parameter family of diffeomorphisms $\Phi_t$ defined in (\ref{canonicalform}) is now the group of diffeomorphisms generated by $-\nabla f$ with $\Phi_0={\rm id}.$ We shall still use $g_t=\Phi_t^*g$ to denote the canonical form of the steady soliton.

Let us fix a point $o\in M$, and we shall impose one more condition on the steady soliton in question, namely, a uniformly bounded Nash entropy:
\begin{equation}
\label{eq: bdd nash}
    \mathcal{N}_{o,0}(\tau)\geq -Y\quad\text{ for all }\quad \tau>0,
\end{equation}
 where $Y\in(0,\infty)$ is a constant and $\mathcal N$ should be regarded as the Nash entropy of the canonical form. We refer the readers to \cite{Bam20a} for the definitions. We shall denote by $|\Omega|_g$ the volume of a measurable subset $\Omega\subset M$ relative to the metric $g$ and by $B_r(x)$ or $B(x,r)$ the geodesic ball centered at $x$ with radius $r$. With these preparation, our main theorem is stated as follows.

\begin{Theorem}
\label{thm: steady vol growth}
Suppose that $(M^n,g,f)$ is a complete steady gradient Ricci soliton normalized as in \eqref{eq: steady def} and the canonical form $(M^n,g_t)_{t\in \mathbb{R}}$ satisfies \eqref{eq: bdd nash}. Additionally, assume that \textbf{either one} of the following conditions is true:
\begin{enumerate}[(1)]
    \item $(M^n,g_t)_{t\in\mathbb{R}}$ arises as a singularity model; or
    \item 
    $(M^n,g)$ has bounded curvature.
\end{enumerate}
Then
\[
c(n,\mu_\infty) r^{\frac{n+1}{2}}\le|B_r(o)|
\le C(n,\mu_\infty)r^n \quad\text{ for all }\quad r>\bar r(n,\mu_\infty),
\]
where $ \mu_\infty:=\inf_{\tau>0} \mathcal{N}_{o,0}(\tau)=\lim_{\tau\to\infty}\mathcal N_{o,0}(\tau)  >-\infty$  and $c(n,\mu_\infty)$ and $C(n,\mu_\infty)$ are positive constants of the form
\[
    c(n,\mu_\infty)
    = \frac{c(n)}{\sqrt{1-\mu_\infty}} e^{\mu_\infty}, \quad C(n,\mu_\infty)=C(n) e^{\mu_\infty}.
\]
Furthermore, the upper bound is also true for all $r>0$ (instead of $r\geq\bar r(n,\mu_\infty)$).
\end{Theorem}

A \emph{singularity model} is an ancient solution $(M^n,g_t)_{t\in(-\infty,0]}$ arising as a blow-up limit of a compact Ricci flow $(\overline{M}^n, \overline{g}_t)_{t\in [0,T)}$ around its singular time. A singularity model in the sense of Hamilton \cite{Ha95} is a smooth Cheeger-Gromov-Hamilton limit, whereas a singularity model in the sense of \cite{Bam20c} is a $\mathbb{F}$-limit (c.f. \cite{Bam20b}). In fact, by \cite[Theorem 2.5]{Bam20c}, a \emph{smooth} singularity model in the sense of \cite{Bam20c} is also a singularity model in the sense of Hamilton (but the reverse is not true). In the assumption of Theorem \ref{thm: steady vol growth}(1), the singularity model can be either in the sense of Hamilton or in the sense of \cite{Bam20c}. 
\\

\noindent\textbf{Remarks.}
\begin{enumerate}
    \item The bounded Nash entropy assumption (\ref{eq: bdd nash}) implies strong noncollapsing, and it is obvious that (\ref{eq: bdd nash}) holds on every singularity model. Yet it is an interesting question to ask whether (\ref{eq: bdd nash}) is equivalent to the (either strong or weak) noncollapsing condition on every steady soliton.
        \item  $\mu_\infty$ in the statement of Theorem \ref{thm: steady vol growth} is the shrinker entropy of any tangent flow at infinity of the ancient solution $(M^n,g_t)_{t\in(-\infty]}$ given by \cite[Theorem 2.40]{Bam20c}. Moreover, the value of $\mu_\infty$ is independent of the choice of the point $o$. This can be seen from \cite[Proposition 4.6]{MZ21}. 
    \item The volume growth upper bound is a direct consequence of \cite[Theorem 8.1]{Bam20a} and we will leave the detailed proof to the reader; in this paper we shall only prove the volume growth lower bound. 
    \item The volume growth upper bound is also sharp since it is satisfied by the steady Gaussian soliton. This conclusion is in the spirit of a similar result in the shrinking case (c.f. \cite{Mun09} and \cite{CZ09}). Previous works on the volume growth upper bound for steady Ricci solitons include \cite{MS13, WW13}. 
    \item It is proved in \cite{CFSZ20} that a $4$-dimensional steady gradient Ricci soliton which arises as a singularity model 
    must have bounded curvature. As a consequence, if $n=4$, then case (2) in the statement of Theorem \ref{thm: steady vol growth} is redundant.
\end{enumerate}

\textbf{Acknowledgements.}
The third named author would like to thank Professor Yuxing Deng for  enlightening discussions regarding the volume growth on steady solitons. The last named author would like to thank Professor Jiaping Wang for some helpful discussions and communications.

\section{Proofs}

Roughly speaking, we prove the main theorem by packing balls centered at $\ell$-centers, namely, the points at which Perelman's \cite{Per02} reduced distance function almost attains its minimum (see below for the definition). Since, as it is shown by the three last-named authors \cite[Proposition 5.6]{CMZ21a}, $\ell$-centers are always close to $H_n$-centers (c.f. \cite[Definition 3.10]{Bam20a}), a ball centered at an $\ell$-center must have a volume lower bound estimate as given by \cite[Theorem 6.2]{Bam20a}. This is the argument which proves the optimal volume growth lower bound. 

Since the canonical form of the steady soliton $(M^n,g,f)$ moves only by diffeomorphism, we may work with Perelman's $\mathcal{L}$-geometry \cite[\S 7]{Per02} on the background of the static manifold $(M^n,g).$ 

\subsection{Perelman's \texorpdfstring{$\mathcal{L}$}{L}-geometry on steady solitons}

\def \tgam {\Tilde{\gamma}}

 As mentioned before, we will use $g_t$ to represent the canonical form of the steady soliton $(M^n,g,f)$ satsifying the conditions of Theorem \ref{thm: steady vol growth}. Recall that Perelman defined the $\mathcal{L}$-length in \cite[\S 7]{Per02}. For any $\tau>0,$ and any piecewisely smooth curve $\Gamma:[0,\tau]\to M$ with $\Gamma(0)=o$, 
\[
\mathcal{L}(\Gamma)
:= \int_0^{\tau} \sqrt{s}(R_{g_{-s}}+|\dot\Gamma|_{g_{-s}}^2)(\Gamma(s))\, ds.
\]
To reinterpret the $\mathcal L$-geometry on the static background $(M,g)$, let
\[
    \gamma(s) = \Phi_{-s}(\Gamma(s))\quad \text{ for }\quad s\in[0,\tau].
\]
Then
\[
    \dot\gamma
    = \nabla f|_{\Gamma} + \Phi_{-s*}(\dot\Gamma),
\]
and
\[
    \mathcal{L}(\Gamma)
= \int_0^{\tau} \sqrt{s}\big(R_g+|\dot\gamma -\nabla f|_g^2\big)(\gamma(s))\, ds,
\]
and this expression only uses the static metric $g.$
If we perform a change of variables: $u=\sqrt{s},$ and write $\tgam(u)= \gamma(u^2),$ then
\[
    \mathcal{L}(\Gamma)
= \int_0^{\sqrt{\tau}} \left(\tfrac{1}{2}\left|\dot\tgam - 2u\nabla f\right|^2 + 2u^2 R(\tgam(u))\right) du.
\]

For any $x\in M$ and $\tau>0,$ we define
\[
    L(\Phi_\tau(x),\tau)
    := \inf_{\Gamma} \mathcal{L}(\Gamma),
\]
where the infimum is taken over all $\Gamma:[0,\tau]\to M$ with $\Gamma(0)=o$ and $\Gamma(\tau)=\Phi_{\tau}(x).$ On the static metric background, we may define an equivalent function:
\begin{equation}\label{static reduced length}
    \Lambda(x,\tau) := L(\Phi_\tau(x),\tau)
    = \inf \int_0^{\tau} \sqrt{s}(R_g+|\dot\gamma -\nabla f|_g^2)(\gamma(s))\, ds,
\end{equation}
where the infimum is taken over all $\gamma:[0,\tau]\to M$ with $\gamma(0)=o$ and $\gamma(\tau)=x$, and a curve at which the above infimum is attained shall be called a \textbf{$\Lambda$-geodesic}.
Accordingly, define
\[
    \lambda(x,\tau) := \ell(\Phi_\tau(x),\tau)
    := \frac{1}{2\sqrt{\tau}}\Lambda(x,\tau).
\]
Arguing as Perelman in \cite[Section 7.1]{Per02}, we have that, for any $\tau>0,$ there is a point $p_\tau\in M$ such that
$\lambda(p_\tau,\tau)=\ell(\Phi_{\tau}(p_\tau),\tau) \le n/2.$
Any such point $p_\tau$ is called an \textbf{$\ell$-center} at time $-\tau$. Note that in our current case we are considering the $\ell$-center on a static metric background, hence it differs from the $\ell$-center defined in \cite{CMZ21a} by a diffeomorphism.

\subsection{Locations of \texorpdfstring{$\ell$}{ell}-centers}

\begin{Lemma}
\label{lem: l dist at o} $\lambda(o,\tau)\ge\tau/12,$ for any $\tau>0.$
\end{Lemma}
\begin{proof}
Let $\gamma:[0,\tau]\to M$ be a loop at $o$ and let $\tgam:[0,\sqrt{\tau}]\to M$ be the reparametrization: $\tgam(u)=\gamma(u^2).$
Then
\begin{align}\label{nonsense}
    \int_0^\tau \sqrt{s}
    &(R + |\dot\gamma-\nabla f|^2)
    =\int_0^{\sqrt{\tau}} \left(\tfrac{1}{2}\left|\dot\tgam - 2u\nabla f\right|^2 + 2u^2 R(\tgam(u))\right) du.\\\nonumber
    &= \int_0^{\sqrt{\tau}} 
    \left(\tfrac{1}{2}|\dot\tgam|^2 
    - 2u (f\circ\tgam - f(o))' + 2u^2\right) du \\\nonumber
    &= \frac{2}{3}\tau^{3/2}
    + \int_0^{\sqrt{\tau}} 
    \left(\tfrac{1}{2}|\dot\tgam|^2 
    + 2 (f\circ\tgam(u) - f(o))\right) du,
\end{align}
where in the second equality we have applied (\ref{eq: steady def}). Let $F(u)=f\circ\tgam(u) - f(o)$ and define
\[
L:= \sup_{u\in [0,\sqrt{\tau}] }\  \dist(o,\tgam(u))
=: \dist(o,\tgam(u_1)),
\]
for some $u_1\in [0,\sqrt{\tau}].$
Then we have
\begin{align*}
    \frac{1}{2}\int_0^{\sqrt{\tau}} 
    |\dot\tgam|^2
    &\ge \frac{1}{2}\int_0^{u_1} 
    |\dot\tgam|^2
    + \frac{1}{2}\int_{u_1}^{\sqrt{\tau}} 
    |\dot\tgam|^2\\
    &\ge \frac{L^2}{2} \left(
    \frac{1}{u_1}+\frac{1}{\sqrt{\tau}-u_1}
    \right)
    \ge \frac{2L^2}{\sqrt{\tau}},
\end{align*}
where we have applied the Cauchy-Schwarz inequality (e.g. $L^2\leq (\int_0^{u_1}|\dot\tgam|)^2\leq \int_0^{u_1}|\dot\tgam|^2\cdot \int_0^{u_1} 1^2$). Since $|\nabla f|\le 1$ by (\ref{eq: steady def}), we have
\[
    |F(u)| \le \dist(\tgam(u),o)\le L,\quad \forall \,u\in [0,\sqrt{\tau}],
\]
and thus
\[
\int_0^{\sqrt{\tau}} 
    2 (f\circ\tgam(u) - f(o)) du\geq -2L\sqrt{\tau}.
\]
In summary, we have
\begin{align*}
    \int_0^\tau \sqrt{s}
    &(R + |\dot\gamma-\nabla f|^2)
    \ge \frac{2}{3}\tau^{3/2}
    + \frac{2L^2}{\sqrt{\tau}} - 2 L \sqrt{\tau}\\
    &= \frac{2}{3}\tau^{3/2}
    + \frac{2}{\sqrt{\tau}}(L^2-L\tau)\\
    &= \frac{1}{6}\tau^{3/2}
    + \frac{2}{\sqrt{\tau}}(L-\tau /2)^2 \ge \frac{1}{6}\tau^{3/2},
\end{align*}
and the conclusions follow by taking the infimum on the left hand side.
     
\end{proof}

The following Lemma is straightforward and is similar to the standard triangle inequality; c.f. \cite[\S 4, Claim 3]{CMZ21a}.
\begin{Lemma}
\label{lem: triang}
For any $x,y\in M,\tau>0$ and
any $\delta\in (0,1),$
\[
    \lambda(x,(1+\delta)^2\tau)
        \le \lambda(y,\tau) +  \frac{\dist^2(x,y)}{\delta \tau}
    + 5\delta \tau.
\]
\end{Lemma}
\begin{proof}
Let $\gamma_1:[0,\tau]\to M$ be a minimizing $\Lambda$-geodesic from $o$ to $y$, namely, a curve at which the infimum in (\ref{static reduced length}) is attained.
Let $\tgam_2:[\sqrt{\tau},(1+\delta)\sqrt{\tau}]\to M$ be a minimizing $g$-geodesic from $y$ to $x$ with constant speed. 
Define $\gamma_2:[\tau,(1+\delta)^2\tau]\to M$ by
$\gamma_2(s)=\tgam_2(\sqrt{s}).$
\begin{align*}
    \Lambda(x,(1+\delta)^2\tau)
    &\le \int_0^\tau \sqrt{s}(R+|\dot\gamma_1-\nabla f|^2)(\gamma_1(s))\, ds
    + \int_\tau^{(1+\delta)^2\tau} \sqrt{s}(R+|\dot\gamma_2-\nabla f|^2)(\gamma_2(s))\, ds\\
    & \le \Lambda(y,\tau)
    + \int_{\sqrt{\tau}}^{(1+\delta)\sqrt{\tau}}
        \left(\tfrac{1}{2}|\dot \tgam_2|^2 
        + 2u |\dot \tgam_2| |\nabla f|
        + 2 u^2(R+|\nabla f|^2)\right) du\\
    &\le \Lambda(y,\tau)
    + \int_{\sqrt{\tau}}^{(1+\delta)\sqrt{\tau}}
\left(|\dot \tgam_2|^2 
        + 4 u^2\right) du\\
    & \le \Lambda(y,\tau)
    + \frac{\dist^2(x,y)}{\delta\sqrt{\tau}}
    +  4\frac{(1+\delta)^{3}-1}{3}\tau^{3/2}\\
    &\le \Lambda(y,\tau)
    + \frac{\dist^2(x,y)}{\delta\sqrt{\tau}}
    +  10\delta \tau^{3/2}.
\end{align*}
The conclusion follows by dividing $2(1+\delta)\sqrt{\tau}$ on both sides.
\end{proof}

\begin{Lemma}
\label{lem: center linear growth}
There is a universal constant $\alpha\in (0,1),$ such that
for any $\tau\ge \bar{\tau}(n)$ and any $\ell$-center $p_\tau$, we have
\[
     \dist(p_\tau,o) 
    \ge \alpha \tau.
\]
\end{Lemma}

\begin{proof}
By Lemma \ref{lem: l dist at o} and Lemma \ref{lem: triang}, for any $\delta\in(0,1),$ if $\tau\ge \bar\tau(n,\delta),$ then we have
\begin{align*}
    \frac{(1+\delta)^2\tau}{12} 
    & \le \lambda(o,(1+\delta)^2\tau)
    \le \lambda(p_\tau,\tau) +  \frac{\dist^2(p_\tau,o)}{\delta\tau}
    + 5\delta\tau\\
    & \le  \frac{\dist^2(p_\tau,o)}{\delta\tau}
    + 10\delta\tau,
\end{align*}
where we have used the fact that $\lambda(p_\tau,\tau)\leq\frac{n}{2}$. We may take, e.g., $\delta=10^{-3}$ to obtain the  inequality.



\end{proof}

\begin{Lemma}
\label{lem: modified center}
For any $\tau\ge \bar\tau(n),$ there is $x_\tau \in M$ such that $\dist(x_\tau,o)=\tau$ and $ \lambda(x_\tau,\tau_0)\le C$ for some $\tau_0\in[c\tau,\tau/\alpha],$ where $c>0$ and $C<\infty$ are dimensional constants and $\alpha$ is given by Lemma \ref{lem: center linear growth}.
\end{Lemma}

\begin{proof}
Let $\gamma:[0,\tau/\alpha]\to M$ be a minimizing $\Lambda$-geodesic from $o$ to $p:=p_{\tau/\alpha}.$ 
By  Lemma \ref{lem: center linear growth}, $\dist(p,o)\ge \tau.$
So we can define
\[
    \tau_0 := \sup\{ s\in [0,\tau/\alpha]: \dist(\gamma(s),o)\le \tau \},\quad
    x_\tau := \gamma(\tau_0).
\]
We first show that $\tau_0\ge c\tau$ for some universal constant $c>0$. 
Define $\tgam:\left[0,\sqrt{\tau/\alpha}\,\right]\to M$ by $\tgam(u)=\gamma(u^2).$
Note that, arguing in the same way as (\ref{nonsense}), we have
\begin{align*}
    \frac{1}{2}\int_0^{\sqrt{\tau_0}}  |\dot \tgam|^2
    &\le \Lambda(p,\tau/\alpha)
    + \int_0^{\sqrt{\tau_0}} 2u \langle \dot\tgam, \nabla f \rangle\\
    &\le n\sqrt{\tau/\alpha}
    + \frac{1}{4}\int_0^{\sqrt{\tau_0}}|\dot \tgam|^2
    + 4 \int_0^{\sqrt{\tau_0}} u^2,\\
   \frac{1}{4}\int_0^{\sqrt{\tau_0}}  |\dot \tgam|^2 &\le n\sqrt{\tau/\alpha} + \frac{4}{3}\tau_0^{3/2}.
\end{align*}
It follows that
\begin{align*}
    \frac{1}{4}\tau^2
    &= \frac{1}{4} \dist(o,x_\tau)^2
    \le \frac{1}{4} \left(
    \int_0^{\sqrt{\tau_0}} |\dot \tgam|
    \right)^2
    \le \frac{1}{4}\sqrt{\tau_0} \int_0^{\sqrt{\tau_0}}  |\dot \tgam|^2\\
    &\le  n\sqrt{\tau_0\tau/\alpha} + \frac{4}{3}\tau_0^{2}
    \le \frac{1}{8}\tau^2
    +\frac{4}{3}\tau_0^{2},
\end{align*}
if $\tau\ge \bar\tau(n).$ Hence $\tau_0\ge c\tau$ for some dimensional constant $c>0.$ Then
\[
    \lambda(x_\tau,\tau_0)
    \le \frac{\sqrt{\tau/\alpha}}{\sqrt{\tau_0}} \lambda(p,\tau/\alpha)
    \le \frac{n}{2\sqrt{c\alpha}}.
\]
\end{proof}

\begin{Lemma}
\label{main lem}
Suppose that $(M^n,g,f)$ satisfies the assumptions in Theorem \ref{thm: steady vol growth}.
Then for any $\tau\ge \bar\tau(n),$ there is $x_\tau\in M$ such that $\dist(x_\tau,o)=\tau$ and
\[
    |B(x_\tau, \sqrt{A\tau})| \ge c e^{\mu_\infty} \tau^{n/2},
\]
where $A=C_n(1-\mu_\infty), c=c(n)>0.$
\end{Lemma}
\begin{proof}
Let $\nu_t=\nu_{o,0|t}$ be the conjugate heat kernel (c.f. \cite[Definition 2.4]{Bam20a}) based at $(o,0)$ coupled with the canonical form $(M,g_t)$.
Let $x_\tau, \tau_0$ be given by Lemma \ref{lem: modified center}.
Recall that $c\tau\le \tau_0\le \tau/\alpha$ and $\lambda(x_\tau,\tau_0)\le C,$ for some dimensional constants $c,C$ and $\alpha$ is given by Lemma \ref{lem: center linear growth}.
By the proof of \cite[Theorem 6.2]{Bam20a},  it suffices to show that
\begin{equation}
\label{eq: nu concen}
    \nu_{-\tau_0}\left(B_{g_{-\tau_0}}\left(y_\tau, \sqrt{\alpha A\tau_0}\right)\right) 
    \ge 1/2,
\end{equation}
where $y_\tau=\Phi_{{\tau_0}}(x_\tau).$ Because once we can show \eqref{eq: nu concen}, by the proof of \cite[Theorem 6.2]{Bam20a}, we have
\begin{align}\label{anothernonsense}
   \left |B_g\left(x_\tau,\sqrt{A\tau}\right)\right|_g
&\ge\left|B_g\left(x_\tau,\sqrt{\alpha A\tau_0}\right)\right|_g
    =\left|B_{g_{-\tau_0}}\left(y_\tau, \sqrt{\alpha A\tau_0}\right)\right|_{g_{-\tau_0}}\\\nonumber
    &\ge c_n e^{\mu_\infty} \tau_0^{n/2}
    \ge c_n e^{\mu_\infty} \tau^{n/2},
\end{align}
where we used the fact that $\tau/\alpha\ge \tau_0\ge c\tau$ for some dimensional constant $c>0$.
We leave the details of the proof of (\ref{anothernonsense}) to the reader. Note that, by \cite[Proposition 3.3]{CMZ21a}, \cite[Theorem 6.2]{Bam20a} also holds for Ricci flows with bounded curvature on compact intervals.

Now we prove \eqref{eq: nu concen}.
Let $(z,-\tau_0)$ be an $H_n$-center of $(o,0).$
By \cite[9.5]{Per02} and \cite[Theorem 7.2]{Bam20a} (or \cite[Theorem 3.2]{CMZ21a}), we have
\begin{align*}
    (4\pi\tau_0)^{-n/2}e^{-C}&\le (4\pi\tau_0)^{-n/2}e^{-\ell(y_\tau,\tau_0)}
    \le K(o,0\,|\,y_\tau,-\tau_0)\\
    &\le C_n e^{-\mu_\infty} \tau_0^{-n/2}\exp\left(
        - \frac{\dist_{-\tau_0}^2(y_\tau,z)}{9\tau_0}
    \right),
\end{align*}
where $K$ is the fundamental solution to the conjugate heat equation, and we also used Lemma \ref{lem: modified center} and the fact that $\ell(y_\tau,\tau_0)=\lambda(x_\tau,\tau_0)\leq C$.
Hence
\[
    \dist_{-\tau_0}^2(y_\tau,z)
    \le 9(-\mu_\infty+C_n)\tau_0.
\]
We choose $A$ so that 
\[
    \alpha A = 18(-\mu_\infty+C_n) + 10H_n.
\]
By \cite[Proposition 3.13]{Bam20a}, we have
\[
    \nu_{-\tau_0}\left(B_{g_{-\tau_0}}\left(y_\tau, \sqrt{\alpha A\tau_0}\right)\right) 
    \ge 
    \nu_{-\tau_0}\left(B_{g_{-\tau_0}}\left(z, \sqrt{\alpha A\tau_0/2}\right)\right) 
    \ge 1 - \frac{2H_n}{\alpha A}
    \ge \frac{1}{2}.
\]
So we finished the proof of \eqref{eq: nu concen}.

\end{proof}

\subsection{Proof of the main theorem}
Now we are ready to prove Theorem \ref{thm: steady vol growth}.

\begin{proof}[Proof of Theorem \ref{thm: steady vol growth}]
Let $\bar\tau(n)<\infty$ be given by Lemma \ref{main lem}. For each $r>10A+\bar\tau(n),$ we construct a decreasing sequence $r=\tau_1>\tau_2>\cdots > \tau_N>0,$ such that $\tau_N<r/10$ and for $1\le j\le N-1,$
\[
    \tau_{j}-\tau_{j+1}
    = \sqrt{A\tau_j} + \sqrt{A\tau_{j+1}}.
\]
As long as $\tau_j\ge r/10$, the above equation is solvable for positive $\tau_{j+1}$ since the discriminant $A+4(\tau_j-\sqrt{A\tau_j})=4(\sqrt{\tau_j}-\sqrt{A}/2)^2\ge 0$. Since $\tau_j\geq r/10$ and $r>10A$, there is a unique positive solution for $\tau_{j+1}$. Moreover,  $\tau_j-\tau_{j+1}\ge \sqrt{A\tau_j}\ge \sqrt{Ar/10}$, hence we can find a finite positive integer $N$ such that $0<\tau_N< r/10$.
For each $j,$ by Lemma \ref{main lem}, there is $x_j\in M$ such that
$\dist(x_j,o)=\tau_j,$ and
\[
    \left|B\left(x_j,\sqrt{A\tau_j}\right)\right|
    \ge c(n)e^{\mu_\infty} \tau_j^{n/2}.
\]
By the construction of $\{\tau_j\},$ the balls $\left\{B\left(x_j,\sqrt{A\tau_j}\,\right)\right\}_{j=1}^N$ are pairwise disjoint.
It follows that
\begin{align*}
    |B_{2r}(o)|
    &\ge \sum_{j=1}^N 
    \left|B\left(x_j,\sqrt{A\tau_j}\right)\right|
    \ge \sum_{j=1}^N  c(n)e^{\mu_\infty} \tau_j^{n/2}\\
    &\ge  \tfrac{c(n)}{2\sqrt{A}}e^{\mu_\infty}\sum_{j=1}^{N-1}   \tau_j^{\frac{n-1}{2}}
    (\tau_j -\tau_{j+1})\\
    &\ge 
    \tfrac{c(n)}{\sqrt{A}}e^{\mu_\infty}\sum_{j=1}^{N-1}   \int_{\tau_{j+1}}^{\tau_j} \tau^{\frac{n-1}{2}}\, d\tau\\
    &= \tfrac{c(n)}{\sqrt{A}}e^{\mu_\infty}
    \int_{\tau_N}^r \tau^{\frac{n-1}{2}}\, d\tau\\
    &\ge \tfrac{c(n)}{\sqrt{A}} e^{\mu_\infty}r^{\frac{n+1}{2}}.
\end{align*}

\end{proof}


\begin{thebibliography}{CCG{\etalchar{+}}10}





\bibitem[Ap17]{Ap17} Appleton, Alexander. \emph{A family of non-collapsed steady Ricci solitons in even dimensions greater or equal to four.} arXiv preprint arXiv:1708.00161 (2017).


\bibitem[Bam20a]{Bam20a}
Richard~H. Bamler,  \emph{{Entropy and heat kernel bounds on a Ricci flow background}},
  https://arxiv.org/abs/2008.07093 (2020).

\bibitem[Bam20b]{Bam20b}
\bysame, \emph{{Compactness theory of the space of super Ricci flows}},
  https://arxiv.org/abs/2008.09298 (2020).
  
\bibitem[Bam20c]{Bam20c} \bysame, \emph{{Structure theory of non-collapsed limits of Ricci flows}},   https://arxiv.org/abs/2009.03243 (2020).

\bibitem[Bam21]{Bam21} \bysame, \emph{On the fundamental group of non-collapsed ancient Ricci flows},
https://arxiv.org/abs/2110.02254 (2021).


\bibitem[BK17]{BK17} Bamler, Richard H.,  Kleiner, Bruce. \emph{Uniqueness and stability of Ricci flow through singularities.} arXiv preprint arXiv:1709.04122 (2017).





\bibitem[Br20]{Br20} Brendle, Simon. \emph{Ancient solutions to the Ricci flow in dimension $3$}. Acta Mathematica \textbf{225} (2020): 1-102.




 







\bibitem[Cao09]{Cao09} 
Cao, Huai-Dong. \emph{Recent progress on Ricci solitons.} arXiv preprint arXiv:0908.2006 (2009).

\bibitem[CZ09]{CZ09} Cao, Huai-Dong; Zhou, Detang. \emph{On complete gradient shrinking Ricci solitons.} Journal of Differential Geometry \textbf{85} (2010): 175-186.

 \bibitem[CN09]{CN09} Jos\'{e} Carrillo and Lei Ni, \emph{Sharp logarithmic Sobolev inequalities on gradient solitons and applications}, Comm. Anal. Geom. \textbf{17} (2009), 721--753.














\bibitem[CMZ21a]{CMZ21a} Chan, Pak-Yeung, Zilu Ma, and Yongjia Zhang. \emph{Ancient Ricci flows with asymptotic solitons.} arXiv preprint arXiv:2106.06904 (2021).

\bibitem[CMZ21b]{CMZ21b} Chan, Pak-Yeung, Zilu Ma, and Yongjia Zhang. \emph{A uniform Sobolev inequality for ancient Ricci flows with bounded Nash entropy.} arXiv preprint arXiv:2107.01419 (2021).

\bibitem[CMZ21c]{CMZ21c} Chan, Pak-Yeung, Zilu Ma, and Yongjia Zhang. \emph{On Ricci flows with closed and smooth tangent flows.} arXiv preprint arXiv:2109.14763 (2021).










\bibitem[CH21]{CH21} Choi, Beomjun; Haslhofer, Robert, \emph{ A note on blowup limits in 3d Ricci flow}. arXiv:2109.13875 









\bibitem[CCGGIIKLLN10]{RFV3}
Chow, B.; Chu, S.; Glickenstein, D.; Guenther, C.; Isenberg, J.; Ivey, T.; Knopf, D.; Lu, P.; Luo, F.; Ni, L. \emph{The Ricci flow: techniques and applications. Part III. Geometric-Analytic Aspects}, Mathematical Surveys and Monographs,  vol. \textbf{163}, AMS, Providence, RI, 2010.



 
 




\bibitem[CFSZ20]{CFSZ20} Chow, Bennett; Freedman, Michael; Shin, Henry; Zhang, Yongjia, \emph{{Curvature growth of some 4-dimensional gradient Ricci soliton singularity models}}, Advances in Mathematics, \textbf{372} (2020), article number 107303.


\bibitem[Cui16]{Cui16} Cui, Xin, \emph{On curvature, volume growth and uniqueness of steady Ricci solitons}, PhD thesis, Lehigh University, 2016.

\bibitem[EMT11]{EMT11} Enders, Joerg; Reto M\"uller, Reto; Topping, Peter, \emph{On type I singularities in Ricci flow.} Communications in Analysis and Geometry, \textbf{19} (2011): 905-922.


\bibitem[GZ08]{GZ08}Gu, Hui-Ling; Zhu, Xi-Ping. \emph{The existence of Type II singularities for the Ricci
flow on $\mathbb{S}^{n+1}$.} Communications in Analysis and Geometry, \textbf{16} (2008): 467–494.

\bibitem[Ha95]{Ha95} Hamilton, Richard. \emph{The formation of singularities in the Ricci flow}, Surveys in Differential Geometry, \textbf{2} (1995): 7-136

\bibitem[KL17]{KL17} Kleiner, Bruce;  Lott, John. \emph{Singular Ricci flows I.} Acta Mathematica \textbf{219} (2017): 65-134.


\bibitem[Lai20]{Lai20} Lai, Yi. \emph{A family of 3d steady gradient solitons that are flying wings.} arXiv preprint arXiv:2010.07272 (2020).

\bibitem[LW20]{LW20} Li, Yu; Wang, Bing. \emph{Heat kernel on Ricci shrinkers}. Calc. Var. Partial Differential Equations 59 (2020), no. 6, Paper No. 194, 84 pp.







































\bibitem[MZ21]{MZ21} Ma, Zilu and Zhang, Yongjia, \emph{Perelman's entropy on ancient Ricci flows}, Journal of Functional Analysis, \emph{to appear}.



\bibitem[Mun09]{Mun09} Munteanu, Ovidiu. \emph{The volume growth of complete gradient shrinking Ricci solitons}. arXiv:0904.0798 (2009).

\bibitem[MS13]{MS13}
Munteanu, Ovidiu, and Natasa Sesum,
\emph{On gradient Ricci solitons.} Journal of Geometric Analysis 23.2 (2013): 539-561.


\bibitem[MW12]{MW12}
Ovidiu Munteanu, and Jiaping Wang, 
\emph{Analysis of weighted Laplacian and applications to Ricci solitons},
Communications in Analysis and Geometry 20, no. 1 (2012): 55-94.








\bibitem[Per02]{Per02} Perelman, Grisha, \emph{The entropy formula for the Ricci flow and its geometric applications}, arXiv:math.DG/0211159.


\bibitem[Per03a]{Per03a} Perelman, Grisha, \emph{Ricci flow with surgery on three-manifolds}, arXiv:math.DG/0303109.

\bibitem[Per03b]{Per03b} Perelman, Grisha. \emph{Finite extinction time for the solutions to the Ricci flow on certain three-manifolds}, arXiv preprint math/0307245 (2003).





\bibitem[WW13]{WW13} 
Guofang Wei, Peng Wu. 
\emph{On volume growth of gradient steady Ricci solitons.} Pacific Journal of Mathematics 265.1 (2013): 233-241.





\bibitem[Zhq07]{Zhq07} Zhang, Qi S. \emph{A uniform Sobolev inequality under Ricci flow.} International Mathematics Research Notices 2007 (2007).




\end{thebibliography}
\bibliographystyle{amsalpha}

\newcommand{\etalchar}[1]{$^{#1}$}
\providecommand{\bysame}{\leavevmode\hbox to3em{\hrulefill}\thinspace}
\providecommand{\MR}{\relax\ifhmode\unskip\space\fi MR }
\providecommand{\MRhref}[2]{%
  \href{http://www.ams.org/mathscinet-getitem?mr=#1}{#2}
}
\providecommand{\href}[2]{#2}

\bigskip
\bigskip

\noindent Department of Mathematics, University of California Berkeley, CA 94720, USA
\\ Email address: \verb"rbamler@berkeley.edu"
\\

\noindent Department of Mathematics, University of California, San Diego, CA 92093, USA
\\ E-mail address: \verb"pachan@ucsd.edu "
\\

\noindent Department of Mathematics, University of California, San Diego, CA 92093, USA
\\ E-mail address: \verb"zim022@ucsd.edu"
\\

\noindent School of Mathematics, University of Minnesota, Twin Cities, MN 55414, USA
\\ E-mail address: \verb"zhan7298@umn.edu"

\end{document}